\documentclass[12pt,oneside]{amsart}

\usepackage{amssymb}
\usepackage{amsmath}
\usepackage{mathrsfs}
\usepackage{verbatim}
\usepackage{lipsum}
\usepackage[retainorgcmds]{IEEEtrantools}

\usepackage[margin=30 mm,heightrounded=true,centering]{geometry}
 \theoremstyle{plain}
 \newtheorem{theorem}{Theorem}[section]
 \newtheorem{lemma}[theorem]{Lemma}
 \newtheorem{question}[theorem]{Question}
 \newtheorem{corollary}[theorem]{Corollary}
 \newtheorem{proposition}[theorem]{Proposition}
      
 \theoremstyle{definition}
 \newtheorem{definition}[theorem]{Definition}
      
 \theoremstyle{remark}

\makeatletter
\def\@setcopyright{}
\def\serieslogo@{}
\makeatother

\begin{document}

\author{Ran Ji}

\title{On The Martin Boundary of rank $1$ manifolds with nonpositive curvature}

\newcommand{\Addresses}{{
  \bigskip
  \footnotesize

  \textsc{Yau Mathematical Sciences Center,
Tsinghua University, Beijing, China 100084}\par\nopagebreak
  \textit{E-mail:} \texttt{rji@mail.tsinghua.edu.cn}

}}

\begin{abstract}
For a manifold with nonpositive curvature, the Martin boundary is described by the behavior of normalized Green's functions at infinity. A classical result by Anderson and Schoen states that if the manifold has pinched negative curvature, the geometric boundary is the same as the Martin boundary. In this paper, we study the Martin boundary of rank $1$ manifolds admitting compact quotients. It is proved that a generic set in the geometric boundary can be identified naturally with a subset of the Martin boundary. This gives a partial answer to one of the open problems in geometry collected by Yau.
\end{abstract}

\maketitle

\section{Introduction}

In this paper we study compactifications of noncompact Riemannian manifolds.

Consider a  complete, simply connected $n$-dimensional Riemannian manifold $M$ with nonpositive sectional curvature. Fix a base point $p \in M$. It is well known that the exponential map at $p$ induces a diffeomorphism between $T_p M$ and $M$. $M(\infty)$, which is defined as the set of equivalence classes of geodesic rays, can be identified with the unit sphere $\mathbb{S}^{n-1}$. A basic fact is that $\overline{M}=M \cup M(\infty)$ with the `cone topology' is a compactification of $M$ \cite{MR0336648}.

On a non-parabolic manifold $M$, i.e., a manifold that possesses an entire Green's function for the Laplacian, Martin \cite{MR0003919} introduced another way to compactify $M$ by attaching equivalent normalized Green's functions. The set of equivalence classes of normalized Green's functions is called the Martin boundary of $M$ and is denoted by $\mathcal{M}$. The `Martin topology' on $\widetilde{M}=M \cup \mathcal{M}$ is compact and induces the topology on $M$.

In 1985, Anderson and Schoen proved that on a manifold with pinched negative curvature, the Martin boundary of $M$ can be identified with the geometric boundary.

\begin{theorem}(\cite{MR794369})
\label{AS}
Let $M$ be a complete, simply connected Riemannian manifold whose sectional curvature satisfies $-b^2 \leq K_M \leq -a^2<0$. Then there exists a natural homeomorphism $\Phi: \widetilde{M} \to \overline{M}$ from the Martin compactification to the geometric compactification which is the identity on $M$.
\end{theorem}

To prove Theorem \ref{AS}, Anderson and Schoen established the `boundary Harnack inequality', the main tool in the study of Martin boundary, to estimate the growth of positive harmonic functions that vanish continuously at infinity in cones. In 1987, using measure theory, Ancona \cite{MR890161} was able give a simpler proof of the boundary Harnack inequality and generalize it to elliptic weakly coercive operators with measurable and bounded coefficients.

In general, the boundary Harnack inequality does not hold on a manifold containing flat strips. In \cite{MR1847509}, the author constructs certain manifolds with nonpositive curvature on which the boundary Harnack inequality fails. However, it is proved that the Martin boundary is still the same as the geometric boundary. Given an additional assumption that $M$ admits a rank $1$ compact quotient, Yau asked the following question in \cite{MR1216573}.

\begin{question}
\label{Yau}
What is the Martin boundary of the universal cover of a rank $1$ compact manifold with nonpositive curvature?
\end{question}

In accordance with the case of pinched negative curvature, Ballmann \cite{MR990144} proved the solvability of the asymptotic Dirichlet problem for such a manifold. Furthermore, it is shown in \cite{MR1269841} that the Poisson boundary, which can be regarded as a subset of the Martin boundary, is naturally isomorphic to the geometric boundary. There are many other classical results for manifolds with pinched negative curvature which carry over to rank $1$ manifolds admitting compact quotients, see, e.g.,\cite{MR1465601}. Hence it is natural to expect that the geometric boundary coincides with the Martin boundary in this case. When $M$ is $3$-dimensional, we are able to show that this is true for at least a generic set in $M(\infty)$, i.e., a set contains a countable intersection of open and dense sets. To be precise, we prove the following theorem.
\begin{theorem}
\label{main}
Let $M$ be a complete, simply connected $3$-dimensional Riemannian manifold with nonpositive curvature. Suppose that $M$  admits a discrete group of isometries $D$ such that $M/D$ is a compact rank $1$ manifold. Then there exits a generic set $R \subset M(\infty)$ such that for any $\xi \in R$ and $\{y_n\}$ any sequence in $M$ converging to $\xi$ in the cone topology, the normalized Green's functions $h_{y_n}=\dfrac{G(x,y_n)}{G(p,y_n)}$ converges to a unique limiting function $h_\xi$ which is independent of the sequence $\{y_n\}$. Moreover, $h_\xi$ vanishes on $M(\infty) \setminus \{\xi\}$ and therefore $h_\xi \neq h_{\xi'}$ for distinct $\xi,\xi' \in R$.
\end{theorem}

\emph{Remark.} Our result can be generalized to higher dimensional manifolds that contain only isolated flat hypersurfaces. In particular, Theorem \ref{main} holds for the universal cover of an irreducible analytic manifold containing a flat hypersurface.
\\

Our proof is a modification of an argument due to Ancona \cite{MR890161}. The paper is organized as follows. In Section 2 and Section 3 we collect preliminaries for nonpositively curved manifolds and rank $1$ manifolds respectively. In Section 4 we discuss the structure of embedded flat planes in $3$-manifolds and show the hyperbolicity that is needed in the proof of the boundary Harnack inequality. In Section 5 we recall some classical results in potential theory. Section 6 is the main technical part of this paper, we establish the boundary Harnack inequality along geodesics that intersect flat planes transversally. And in Section 7 we show how the boundary Harnack inequality implies Theorem \ref{main}.
\\

\textbf{Acknowledgements} The author would like to thank Professor J\'ozef Dodziuk for the invaluable suggestions and constant encouragement. The author is also grateful to Professor Werner Ballmann for helpful discussions.

\section{Boundaries at infinity}

\subsection{The geometric boundary}

The most natural way to compactify $M$ is by the asymptotic classes of geodesic rays, introduced by Eberlein and O'Neill \cite{MR0336648}.

We say two geodesic rays $\gamma_1, \gamma_2:[0,\infty) \to M$ are equivalent, and denoted it by $\gamma_1 \sim \gamma_2$, if there exists a constant $C$ such that the following inequality
$$d(\gamma_1(t),\gamma_2(t)) \leq C$$ holds for all $t \geq 0$.

Define $M(\infty)$, the sphere at infinity, to be
$$M(\infty)= \text{the set of all geodesic rays}/ \sim.$$ 

Denote by $S_p$ the unit sphere in $T_p(M)$. Given $\upsilon \in S_p$, there exists a unique geodesic ray $\gamma_\upsilon:[0,\infty) \to M$ satisfying $\gamma_\upsilon(0)=p$ and $\gamma'_\upsilon(0)=\upsilon$. It follows from the classical Toponogov comparison theorem that two geodesic rays $\gamma_1$ and $\gamma_2$ starting from $p$ are equivalent  if and only if $\gamma_1=\gamma_2$. On the other hand each equivalence class contains a representative emanating from $p$. Thus $M(\infty)$ can be identified with $S_p$ for each $p \in M$.

Now we can define the cone $C_p(\upsilon, \theta)$ about $\omega$ of angle $\theta$ by
$$C_p(\upsilon, \theta)=\{x \in M:\angle(\upsilon,\gamma'_{px}(0))<\theta\},$$
here $\gamma_{px}$ is the geodesic ray starting from $p$ that passes through $x$. We call
$$T_p(\upsilon,\theta,R)=C_p(\upsilon, \theta) \setminus B_p(R)$$
a truncated cone of radius $R$. We denote $M \cup M(\infty)$ by $\overline{M}$. Then the set of $T_p(\upsilon,\theta,R)$ for all $\upsilon \in S_p$, $\theta$ and $R>0$ and $B_q(r)$ for all $q \in M$ and $r>0$ form a basis of a topology on $\overline{M}$, which is called the cone topology. This topology makes $\overline{M}$ a compactification of $M$.

\subsection{The Martin boundary}

In \cite{MR0003919}, the author introduced another way to compactify non-parabolic manifolds using normalized Green's functions.

Suppose that $M$ admits an entire Green's function $G$ for the Laplacian $\Delta$. For $x,y \in M$, let
\begin{displaymath}
h_y(x)=\dfrac{G(x,y)}{G(p,y)}
\end{displaymath}
be the normalized Green's function with $h_y(p)=1$. A sequence $Y=\{y_i\}$ is called fundamental if $h_{y_i}$ converges to a positive harmonic function $h_Y$ on $M$. Two fundamental sequences $Y$ and $\overline{Y}$ are said to be equivalent if the corresponding limiting harmonic functions $h_Y$ and $h_{\overline{Y}}$ are the same.

\begin{definition}
The Martin boundary $\mathcal{M}$ of $M$ is the set of equivalence classes of non-convergent fundamental sequences.
\end{definition}

Let $\widetilde{M}=M \cup \mathcal{M}$. For every $y \in M$, all sequences  converging to $y$ correspond to an equivalence class $[Y]$. On the other hand, two fundamental sequences that have different limit points in $M$ are not equivalent. Thus $\widetilde{M}$ can be identified with the set of  equivalence classes of fundamental sequences. Define a metric $\rho$ on $\widetilde{M}$
\begin{equation*}
\rho([Y],[Y'])=\sup_{B_p(1)} |h_Y(x)-h_{Y'}(x)|
\end{equation*}
for $[Y],[Y'] \in \widetilde{M}$. The topology induced by $\rho$ makes $\widetilde{M}$ a compactification of $M$.

\emph{Remark.} In contrast with the case when $M$ has negative pinched curvature, the Liouville theorem says that a positive harmonic function on $\mathbb{R}^n$ must be constant. Hence the only positive limiting function is $h=1$ and the Martin compactification of $\mathbb{R}^n$ is homeomorphic to the one point compactification $\mathbb{S}^n$.

\section{Rank $1$ manifolds}

In this section we give a quick overview of rank $1$ manifolds. Unless otherwise stated, all results can be found in \cite{MR656659}.

Let $\gamma$ be a complete geodesic in $M$. Recall that $\mathrm{rank}(\gamma)$ is the dimension of the space of all parallel Jacobi fields along $\gamma$. We can define the rank of $M$ to be
\begin{equation*}
\mathrm{rank}(M)=\min \{\mathrm{rank}(\gamma_\omega),\omega \in S M\},
\end{equation*}
where $\gamma_\omega$ denotes the complete geodesic with initial velocity $\omega$. By definition, $1 \leq \mathrm{rank}(M) \leq n$. In particular, a manifold is of rank $1$ if it admits a geodesic which doesn't bound an infinitesimal flat strip. 

The following lemma is essential in the study of rank $1$ manifolds. 
\begin{lemma}
\label{rigidity}
Let $M$ be a complete, simply connected manifold with nonpositive curvature. If $g$ is a geodesic of rank $1$ and $\{h_n\}$ is a sequence of geodesics such that $h_n(\infty) \to g(\infty)$ and $h_n(-\infty) \to g(-\infty)$ in the cone topology, then $\{h_n\}$ converges to $g$ in the sense that 
\begin{equation*}
\mathrm{dist}({g}'(0), {h}'_n) \to 0 \text{ as } n \to \infty,
\end{equation*}
where $\mathrm{dist}$ denotes the distance in the tangent bundle $TM$.
\end{lemma}

We are mainly interested in the case that there is a discrete group of isometries $D$ acting freely on $M$ such that $M/D$ is compact. It is known that every isometry $\phi$ in $D \setminus \{Id\}$ is axial, i.e., there exists a geodesic $\gamma$ and a constant $T>0$ such that $\phi(\gamma(t))=\gamma(t+T)$ for all $t \in \mathbb{R}$. $\gamma$ is called an axis of $\phi$.
\begin{definition}
An axial isometry $\phi$ is called hyperbolic if it has an axis $\gamma$ with $\mathrm{rank}(\gamma)=1$.
\end{definition}

This definition is motivated by following theorem, which is well known if the sectional curvature $K_M$ is bounded from above by a negative constant.

\begin{theorem}
\label{hyperbolic}
Let $\phi$ be a hyperbolic axial isometry and $\gamma$ be an axis of $\phi$. Then for any neighborhoods $U$ of $\gamma(-\infty)$ and $V$ of $\gamma(\infty)$, there exists $N \in \mathbb{N}$ such that $\phi^N(M\setminus U) \subset V$.
\end{theorem}

We will also use the following properties of hyperbolic axial isometries. Recall that $D$ is a group of isometries of $M$ such that $M/D$ is a compact manifold.

\begin{theorem}
\label{density}
Suppose $g$ is a geodesic of rank $1$. For any neighborhoods $U$ of $g(-\infty)$ and $V$ of $g(\infty)$, there exists an axis $\gamma$ of a hyperbolic axial isometry $\phi \in D$ such that $\gamma(-\infty) \in U$ and $\gamma(\infty) \in V$.
\end{theorem}
\begin{theorem}
\label{free}
The isometry group $D$ contains a free subgroup on $2$ generators.
\end{theorem}
\begin{theorem}
\label{density}
Denote by $\mathrm{Per}_{\mathrm{hyp}}(D)$ the set of vectors in $\mathrm{S}M$ which are tangent to axis of hyperbolic isometries of $D$. Then $\mathrm{Per}_{\mathrm{hyp}}(D)$ is dense in $\mathrm{S}M$.
\end{theorem}

Combining Theorem \ref{free} with the following theorem due to Brooks, we obtain immediately that $\lambda_1(\Delta)$, the first eigenvalue of the Laplacian operator on $M$, is positive.

\begin{theorem}(\cite{MR634438})
Let $D$ be a group of isometries acting on $M$ such that $M/D$ is a compact manifold. Then $D$ is amenable if and only if $0$ is in the spectrum of the Laplacian.
\end{theorem}

\section{flats in $3$-manifolds}

In this section, we study embedded flat planes in the universal cover of a rank $1$ compact $3$-manifold and derive the hyperbolicity of embedded rectangles which intersects a flat plane transversally.

Let $M$ be a complete, simply connected $3$-dimensional Riemannian manifold with nonpositive curvature. We further assume that there is a group of isometries D acting on $M$ such that $M/D$ is a rank $1$ compact  manifold.

If $M$ doesn't contain a flat plane, then $M$ satisfies the uniform visibility axiom \cite{MR0295387}. In this case the Martin boundary is naturally identified with the geometric boundary. The proof is due to Ancona \cite{MR890161} in the case of negative curvature and it carries over easily to uniform visibility manifolds. Ancona's argument depends on the geometric fact that all geodesics in a uniform visibility manifold are equi-hyperbolic, i.e. given $0<\theta<\pi$, there exists constants $\theta_0=\theta_0(\theta,M)>0$ and $T=T(\theta,M)>0$ such that for any geodesic $\gamma$ in $M$,
\begin{equation*}
C_{\gamma(T)}(\gamma'(T),\theta+\theta_0) \subset C_{\gamma(0)}(\gamma'(0),\theta).
\end{equation*}
However, this property is not true for manifolds admitting flatness. In fact, even for geodesics in a small neighborhood of a hyperbolic axis, we don't have the equi-hyperbolicity. Therefore we need to study the divergence for geodesics intersecting a flat plane transversally.

The following Lemma gives the structure of flat planes in $M$. A large part of the argument in the proof is due to Schroeder \cite{MR1050413}.

\begin{lemma}
\label{half flat}
Suppose that $M$ contains a flat plane. Then there exists a flat plane $F$ in $M$ such that
\begin{enumerate}
\item[(1)] The set $\{\phi F| \phi \in D\}$ is discrete.

\item[(2)] There is at least one component of $M \setminus F$, denoted by $M^+$, such that none of the geodesics in $F$ bound a flat half plane in $M^+$. 
\end{enumerate}
\end{lemma}

\begin{proof}
Assume that $F_1$ is a flat plane in $M$ and $H$ is a complete geodesic in $F_1$ which bounds a flat half plane $F_2^+ \not\subset F_1$, then $H=F_1 \cap F_2^+$. Denote by $P_H$ the union of totally geodesic submanifolds parallel to $H$. It is known that $P_H$ is a closed convex subset of $M$ and splits isometrically as $H \times N$(see \cite{MR823981} for the details). $P_H$ contains the convex hull of $F_1 \cup F_2^+$ and therefore has dimension $3$. $P_H$ cannot be complete, otherwise $M=P_H$ is reducible and $\mathrm{rank}(M) \geq 2$. Thus $N$ has nonempty boundary. Let $p \in \partial N$ and $g$ be a complete geodesic in $N$ with minimal distance to $p$. Now let $F$ be the flat plane $H \times g \subset H \times N$. $F$ divides $M$ into two components. Denote by $M^+$ the component of $M \setminus F$ that doesn't contain $H$ and by $M^-$ the other component. We claim that $F$ and $M^+$ satisfy the required properties.

For the proof of the discreteness of $\{\phi F: \phi \in D\}$, see Lemma 5 in \cite{MR1050413}.

It remains to show (2). we first claim that $F$ doesn't intersect any other flat plane in $M$. Indeed, it is sufficient to consider the case that the intersecting geodesic is not parallel to $H$, otherwise it contradicts the minimality of $d(p,g)$. Thus there are two intersecting complete geodesics in $F$ such that each of them bounds a flat half plane which lies completely in $M^-$. Then $M$ contains a $3$-dimensional flat half space, which contradicts $\mathrm{rank}(M)=1$.

Now assume the contrary of (2) that there is a complete geodesic $g'$ in $F$  that bounds a flat half plane in $M^+$. Then every geodesic in this flat half plane parallel to $g'$ has endpoints in $F(\infty)$. This implies that either $p \notin \partial N$ or $g$ is not the closest complete geodesic to $p$. Lemma \ref{half flat} is proved by contradiction.
\end{proof}
Let $F$, $M^+$ be as in Lemma \ref{half flat}. From now on we fix $p \in F$. Let $NF=\coprod_{q \in F} N_q M$ be the normal bundle of $F$. At each point of $F$ there are exactly two unit normal vectors. We denote by $N$ the unit normal vector field on $F$ pointing to $M^+$ and by $N^+F$ all nonzero vectors in $NF$ with the same orientation. Since $F$ is a closed and totally geodesic hypersurface, the exponential map $\mathrm{exp}:N^+F \to M^+$ is a diffeomorphism.

Let $\gamma:[0,S] \to M$ be a geodesic segment parametrized by arc length. Let $X$ be a vector field along $\gamma$ with $<X,\gamma'>=0$ and $|X|=1$. Consider the parametrized half strip
\begin{eqnarray}
\nonumber \Sigma: [0,S] \times [0,\infty) &\to & M \\
\nonumber (s,t) &\mapsto & \mathrm{exp}_{\gamma(s)} \, tX(s).
\end{eqnarray}
Let $L(t)$ be the length of the curve $\Sigma(\cdot, t)$. It follows from the second variation formula of arc length(see \cite{MR0251664}) that
\begin{equation*}
L'(t)= \int_{\Sigma([0,S] \times [0,t])} \dfrac{\| \nabla_{{\partial \Sigma}/{\partial t}} \frac{\partial \Sigma}{\partial s}\|^2} {\|\frac{\partial \Sigma}{\partial t} \wedge \frac{\partial \Sigma}{\partial s}\|^2} - K(\frac{\partial \Sigma}{\partial s},\frac{\partial \Sigma}{\partial t}) \, \mathrm{d}A
\end{equation*}
for all $t\in (0, \infty)$. We see immediately that $L$ is convex and non-decreasing in $t$.

Since $M$ admits a compact quotient, we have the following consequence of Lemma \ref{half flat}.
\begin{lemma}
\label{rectangle}
There exist constants $R, \kappa,\zeta,\tau>0$, such that for any geodesic segment $\gamma:[0,R] \to F$ parametrized by arc length and the embedded surface
\begin{eqnarray*}
\Sigma: [0,R] \times [0,\infty) &\to & M \\
(s,t) &\mapsto & \mathrm{exp}_{\gamma(s)} \, tN(s)
\end{eqnarray*}
we have:

\begin{enumerate}
\item[(\textrm{1})] the total Gaussian curvature of the embedded square $\Sigma([0,R]\times[0,R])$ is less than $-\kappa$, i.e., 
\begin{displaymath}
\int_{\Sigma([0,R]^2)} -\dfrac{\| \nabla_{{\partial \Sigma}/{\partial t}} \frac{\partial \Sigma}{\partial s}\|^2} {\|\frac{\partial \Sigma}{\partial t} \wedge \frac{\partial \Sigma}{\partial s}\|^2} + K(\frac{\partial \Sigma}{\partial s},\frac{\partial \Sigma}{\partial t}) \, \mathrm{dA} \leq -\kappa.
\end{displaymath}

\item[(\textrm{2})]  Let $\square_\Sigma$ be the geodesic quadrangle in $M$ with vertices $\Sigma(0,0),\Sigma(0,R),\Sigma(R,R),\Sigma(R,0)$. Then the sum of the interior angles of $\square_\Sigma$ is less than $2\pi-\zeta$.

\item[(\textrm{3})]  For $t \geq R$ we have $d(\Sigma(0,t),\Sigma(R,t)) \geq R+\tau(t-R).$ 
\end{enumerate}
\end{lemma}

\begin{proof}
(1) Assuming the contrary, then for any $R \in \mathbb{R}$ we can find a sequence of embedded squares $\Sigma_i:[0,R]^2 \to M$ such that the total Gaussian curvature of $\Sigma_i$ is less than $1/i$, $i=1,2,\cdots$. By Lemma \ref{half flat} we can choose a sequence $\phi_i \in D$ so that $\phi_i(\Sigma_i)$ converges to an embedded flat square $[0,R]^2$ perpendicular to $F$ which has one side on $F$. Since the component of $M \setminus F$ that doesn't contain $[0,R]^2$ must contain a flat half plane with boundary on $F$, we have $ [0,R]^2 \subset \overline{M^+}$. Taking $R \to \infty$ we get a flat half plane in $M^+$, which contradicts Lemma \ref{half flat}.

(2) and (3) then follow easily from (1) and a standard convexity argument.
\end{proof}

Next we want to show an analogous result for embedded sectors. To do this we need the notion of spherical distance.
\begin{definition}
For $x,y \in S_p(r),$ the spherical distance $d^S(x,y)$ is defined as the length of the shortest curve in $S_p(r)$ from $x$ to $y$.

For $x \in S_p(r)$ and $A \subset M$ which has nontrivial intersection with $S_p(r)$, we use $d^S(x,A)$ to denote
\begin{displaymath}
d^S(x,A)=\inf\{d^S(x,y)| y \in A \cap S_p(r)\}.
\end{displaymath}

For two subsets $A,B \subset M$, we denote
\begin{displaymath}
d^S(A,B)=\inf\{d^S(x,y)| x \in A \cap S_p(r), y\in B \cap S_p(r)\}.
\end{displaymath}
Sometimes we use $d^r(A,B)$ instead of $d^S(A,B)$ to specify the sphere $S_p(r)$.
\end{definition}

We can estimate the growth of spherical distance by the standard Jacobi field argument.

\begin{lemma}
\label{spherical growth}
There exists a constant $b>0$ depending only on the lower bound of $K_M$ such that for any $v,w \in S_p(M)$ and $0<t' \leq t$, 
\begin{equation*}
\dfrac{t}{t'} \leq \dfrac{d^S(\gamma_v(t),\gamma_x(t))}{d^S(\gamma_v(t'),\gamma_x(t'))} \leq e^{b(t-t')}.
\end{equation*}
\end{lemma}

\begin{lemma}
\label{spherical standard}
Let $x,y$ be two different points on $S_p(r)$ with $r>2$. Suppose that $d^S(x,y)=s \leq 1$. Then the distance from $x$ to the geodesic $\gamma_{py}$ satisfies
$$d(x,\gamma_{py}) \geq s',$$
where $s'$ is a constant depending only on $s$ and $M$.
\end{lemma}
\begin{proof}
Let $g$ be the geodesic segment  minimizing $d(x,\gamma_{py})$ with $g(0)=x$ and $g(1)=\tilde{y} \in \gamma_{py}$. $g$ lies completely outside $B_p(d(p,\tilde{y}))$. Let $\tilde{x}=S_p(d(p,\tilde{y})) \cap \gamma_{py}$. We have $d(p,\tilde{x})=d(p,\tilde{y}) \geq r-2s$. Then Lemma \ref{spherical growth} implies
\begin{equation}
\label{spherical standard 1}
d^S (\tilde{x},\tilde{y}) \geq s'
\end{equation} 
for a constant $s'$ independent of $x,y$ and $r$.

For $t \in [0,1]$, let $v(t) \in S_p M$ be the unit initial velocity of the geodesic segment from $p$ to $g(t)$. Consider the embedded sector
\begin{eqnarray*}
\Gamma: [0,r] \times [0,1] &\to & M \\
(u,t) &\to & \mathrm{exp}_p \, u v(t).
\end{eqnarray*}
The induced Gaussian curvature of $\Gamma$ is nonpositive. Therefore
$$d(x,\gamma_{py})=L(g(\cdot)) \geq l(\Gamma(d(p,\tilde{y}),\cdot)) \geq d^S (\tilde{x},\tilde{y}),$$ 
which together with \eqref{spherical standard 1} proves the lemma.

\end{proof}

\emph{Remark.} The growth of spherical distance is closely related to the Tits metric on $M(\infty)$, which reflects the flatness of $M$ at infinity.
\\

Set $B^F(r)=B_p(r) \cap F$ and $F_{R}(r)=\{ \mathrm{exp}_q(R N(q)) | q \in B^F(r)\}$. Define the cone of $F_{R}(r)$ by
$$C_p(F_{R}(r))=\{x \in M^+ | \gamma_{px} \cap F_{R}(r) \neq \emptyset\},$$ where $\gamma_{px}$ denotes the geodesic ray starting from $p$ that passes through $x$.

With the notion of spherical distance, we have the following sector version of Lemma \ref{rectangle}.
\begin{lemma}
\label{sector}
There exists constants $R_0 , \eta>0$ such that
\begin{enumerate}
\item[(\textrm{1})] Suppose that $x \in C_p(F_{R_0}(jR_0)), y \in M\setminus C_p(F_{R_0}((j+1)R_0))$ for some $i \in \mathbb{N}$ and $d(x,F),d(y,F) \geq R_0$. Let $\triangle pxy$ be the geodesic triangle in $M$ with vertices $p,x,y$. Then the sum of the interior angles $\angle x$ and $\angle y$ is less than $\pi-\eta$.

\item[(\textrm{2})] For  $j \in \mathbb{N}$ and $r \geq j R_0$ we have
$$d^r(C_p(F_{R_0}(jR_0)), M\setminus C_p(F_{R_0}((j+1)R_0))) \geq \eta(r-j R_0).$$
\end{enumerate}
\end{lemma}

\section{Coercive Operators}

In this section, we recall some classical results in potential theory that will be needed in the proof of Theorem \ref{main}. We assume that $M$ is a complete, simply connected rank $1$ manifold which admits a compact quotient. Recall that $M$ has positive first eigenvalue for the Laplacian and therefore is non-parabolic. We point out that all results of this section were obtained for manifolds with pinched curvature and positive first eigenvalue in \cite{MR890161}. However, some of the proofs are simpler if we have compactness.

\begin{definition}
A self adjoint elliptic operator $\mathscr{L}$ on $M$ is called coercive if $\lambda_1(\mathscr{L})$, the first eigenvalue of $\mathscr{L}$, is positive. 
\end{definition}
  
We consider elliptic operators in the form $\Delta+tI$. If $\lambda_1(\Delta)>0$, it is well known that there exists $\varepsilon>0$, such that for any $t \in (-\infty,\varepsilon]$, $\Delta+tI$ is coercive and admits an entire Green's function $G^t(x,y)$ satisfying
\begin{equation*}
(\Delta+tI) G^t(x,y) = -\delta_x(y)
\end{equation*}
for all $x,y \in M$, where $\delta_x$ is the Dirac measure at $x$ and $\Delta$ is applied to the variable $y$.

Let $\Omega$ be a domain in $M$ with piecewise smooth boundary. There also exists an entire Green's function $G^t_\Omega$ for $\Delta+tI$ satisfying the Dirichlet boundary condition on $\partial \Omega$.

For a positive measure $\mu$ on $M$ we denote by $G^t(\mu)$ the function
\begin{equation*}
G^t(\mu)(x)= \int G^t(x,y)d\mu(y).
\end{equation*}
If $G^t(\mu)$ is not identically $+\infty$, $G^t(\mu)$ is the only potential satisfying $(\Delta+tI)G^t(\mu)=-\mu$. We have
\begin{equation}
\label{Green comparison}
G^{t}(\mu)=G(\mu)+tG(G^{t}(\mu)).
\end{equation}

By compactness we obtain the following lemma immediately.
\begin{lemma}
There is a constant $c=c(M,t)$ such that if $x,y \in M$ and $d(x,y)=1$, \begin{equation*}
\dfrac{1}{c} \leq G^t(x,y) \leq c.
\end{equation*}
\end{lemma}

We will need the classical gradient estimate due to Yau \cite{MR0431040}. 
\begin{theorem}
Suppose that $u$ is a positive $(\Delta+tI)$-harmonic function on $B_x(r)$, i.e. $(\Delta+tI)u=0$.  Then there exists $c=c(M,t,r)$ such that
\begin{equation*}
\dfrac{|\nabla u|}{u} \leq c\text{ on } B_x(r/2).
\end{equation*}
\end{theorem}

An immediate consequence is the Harnack inequality. 
\begin{corollary}
Suppose that $u$ is a positive $(\Delta+tI)$-harmonic function on $B_x(r)$. Then for all $y\in B_x(r/2)$,
\begin{equation*}
\dfrac{1}{c} u(x) \leq u(y) \leq c u(x),
\end{equation*}
where $c$ is a constant that depends only on $M$, $t$ and $r$.
 \end{corollary}

For a bounded domain $\Omega \subset M$ and $x \in \Omega$, we denote by $\mu_x^t$ the $(\Delta+tI)$-harmonic measure of $\Omega$. Let $g_x^t(y)$ be the continuous function which equals to $G_\Omega^t(x,y)$ for $y \in \Omega$ and vanishes for $y \notin \Omega$. We have the following relation between $\mu_x^t$ and $g_x^t$.
\begin{equation*}
(\Delta+tI)g_x^t=-\delta_x+\mu_x^t.
\end{equation*}

\begin{lemma}
\label{measure comparison}
For every $t<\varepsilon$ and every $r>0$, there exists a constant $\delta=\delta(M,t,r)$, $0<\delta<1$, such that for every ball $B_x(r)$ in $M$, the $\Delta$-Green's function $g_x$ of $B_x(r)$, and the $(\Delta+t I)$-Green's function $g^{t}_x$ of $B_x(r)$ satisfy
\begin{equation*}
g_x(y) \leq (1-\delta) g^{t}_x(y)
\end{equation*}
for all $y \in M \setminus B_x(r/4)$.

Moreover, the $\Delta$-harmonic measure $\mu_x$ of $B_x(r)$, and the $(\Delta+tI)$-harmonic measure $\mu^{t}_x$ satisfy
\begin{equation*}
\mu_x \leq (1-\delta) \mu^{t}_x.
\end{equation*}
\end{lemma}

\begin{proof}

It follows from the Harnack inequality that $g(z,y) \geq c$ for $y,z \in B_x(r/2)$, where $c$ can be taken independent of $x$. Therefore by \eqref{Green comparison} we have
\begin{eqnarray*}
g^{t}_x(y) &=& g_x(y)+t \int_{B_x(r)} g^t_x(z) g(y,z) \, \mathrm{dvol}(z)\\
&\geq & g_x(y) + ct \int_{B_x(r/2)\setminus B_x(r/4)}g^{t}_x(z) \, \mathrm{dvol}(z)\\
\end{eqnarray*}
By the volume comparison theorem, $\mathrm{Vol}(B_x(r/2)\setminus B_x(r/4))$ is uniformly bounded from below by a constant $c'=c'(M,r)$. Also for $y,z \in B_x(r/2)\setminus B_x(r/4)$, we have $g^{t}_x(z) \geq c''g^{t}_x(y)$ with $c''=c''(M,r)$ by the Harnack inequality. Therefore by taking $\delta=cc'c''t$ we obtain
\begin{equation}
\label{g^t}
(1-\delta)g^{t}_x(y) \geq g_x(y)
\end{equation}
for $y \in B_x(r/2)\setminus B_x(r/4)$. Since $g^{t}_x$ is a superharmonic function on $B_x(r)\setminus B_x(r/4)$, by the maximum principle, \eqref{g^t} holds for all $y \in M \setminus B_x(r/4)$.

The second part of the lemma is an immediate consequence of the following lemma.
\end{proof}

\begin{lemma}
\label{funtion-measure}
Let $\Omega$ be a bounded domain in $M$ with smooth boundary and $x \in \Omega$. Denote by $g_x$ the $\Delta$-Green's function of $\Omega$ and $g^t_x$ the $(\Delta+tI)$-Green's function of $\Omega$. If for some positive number $k$, $g_x \leq k g^t_x$ outside some compact subset of $\Omega$, then $\mu_x \leq k \mu^{t}_x$, where $\mu_x$ is the $\Delta$-harmonic measure of $\Omega$ and $\mu^{t}_x$ is the $(\Delta+tI)$-harmonic measure of $\Omega$.
\end{lemma}

\begin{corollary}
\label{measure comparison 2}
Let $t,\delta$ be as in Lemma \ref{measure comparison}. We have for all $x,y \in M$, there is a constant $c=c(M)$ such that
\begin{equation}
\label{G growth}
G(x,y) \leq c(1-\delta)^{d(x,y)}G^{t}(x,y)
\end{equation}

Moreover, the $\Delta-$harmonic measure $\mu_{B_x(R)}$  of $B_x(R)$, and the $(\Delta+t I)$ harmonic measure $\mu^{t}_{B_x(R)}$ satisfy
\begin{equation*}
\mu_{B_x(R)} \leq c(1-\delta)^{R} \mu^{t}_{B_x(R)}.
\end{equation*}
\end{corollary}
\begin{proof}
We prove $G(x,y) \leq (1-\delta)^{k-1}G^{t}(x,y)$ by induction on $k=d(x,y)\in \mathbb{N}$.

When $k=1$, the statement is trivial since $G(x,y) \leq G^{t}(x,y)$.

Assume the estimate holds for $d(x,z) =k$, which implies that 
\begin{equation}
\label{G growth k}
G_x(z) \leq (1-\delta)^{k-1}G^{t}_x(z)
\end{equation}
for $z \in \partial B_x(k)$. It follows from the maximum principle that \eqref{G growth k} holds for all $z \in M \setminus B_x(k)$.

Now consider $y \in M$ such that $d(x,y)=k+1$. Thus $\partial B_y(1)$ is contained in $M \setminus B_x(k)$. Let $\mu_y, \mu^{t}_y$ be the $\Delta$-harmonic and $(\Delta+t I)$-harmonic measures respectively. By Lemma \ref{measure comparison} we have
\begin{eqnarray*}
G(x,y)&=&\int_{\partial B_y(1)} G(x,z) \mathrm{d}\mu_y(z)\\
&\leq&  (1-\delta)\int_{\partial B_y(1)} G(x,z) \mathrm{d}\mu^{t}_y(z)\\
&\leq& (1-\delta)(1-\delta)^{k-1} \int_{\partial B_y(1)} G^{t}(x,z) \mathrm{d}\mu^{t}_y(z)\\
&=&(1-\delta)^{k} G^{t}(x,y).
\end{eqnarray*}

It follows from the Harnack inequality that there exists a constant $c=c(M)$ such that \eqref{G growth} holds for all $d(x,y) >0$.

The second part of the corollary follows from Lemma \ref{funtion-measure} immediately.
\end{proof}

To define the harmonic measure of a not necessarily bounded domain, we need the concept of reduction introduced by Brelot in \cite{MR0281940}.

\begin{definition}
Let $x \in \Omega$. $R^{\overline{\Omega}^c}_{G_x}$, the reduction of $G_x$ on $\overline{\Omega}^c$, is defined as follows,
\begin{equation*}
R^{\overline{\Omega}^c}_{G_x}=\inf \{s:M \to \mathbb{R}~| ~s>0 \text{ and superharmonic on } M; s \geq G_x \text{ on } \overline{\Omega}^c \}.
\end{equation*}

\end{definition}

$R^{\overline{\Omega}^c}_{G_x}$ is an $\Delta$-potential, and if we put $\nu_x=-\Delta(R^{\overline{\Omega}^c}_{G_x})$, then $\nu_x$ is the positive harmonic measure with support on $\partial \Omega$ such that
\begin{equation*}
G(x,y)=\int G(z,y) d \nu_x(z).
\end{equation*}

\section{Boundary Harnack Inequalities}

In the sequel we always assume that $p\in F$ is a fixed point and $\upsilon$ is the unit normal vector at $p$ pointing to $M^+$,  where $F$, $M^+$ are as in Lemma \ref{half flat}.  We first prove the following weak version of the boundary Harnack inequality.

\begin{proposition}
\label{proposition}
For any $\theta \in (0,\pi/2)$, there exists a constant $c>0$ depending only on $R_0, \theta, \varepsilon$ and $M$ such that
\begin{displaymath}
G(x,y) \leq c G(x,p)G^\varepsilon(p,y)
\end{displaymath}
for all $x \in T_p(-\upsilon,\pi/2,1)$ and $y \in T_p(\upsilon,\theta,1)$.   
 \end{proposition}

Recall that $R_0$ is chosen to satisfy the assumption in Lemma \ref{sector} and $C_p(F_{R_0}(jR_0))$ is the cone of $F_{R_0}(jR_0)=\{ \mathrm{exp}_q(R_0 N(q)) | q \in B^F(jR_0)\}$. Denote simply $C_j=C_p(F_{R_0}(jR_0))$ and $T_j=C_j \setminus B_p(1)$. Observe that given $\theta\in (0,\pi/2)$, $C_p(\upsilon,\theta) \subset C_k$ for $k \in \mathbb{N}$ large enough. Thus Proposition \ref{proposition} is an immediate consequence of the following lemma.

\begin{lemma}
\label{Harnack cone}
Let $x$ be an arbitrary point in $T_p(-\upsilon,\pi/2,R_0)$. Choose $k \in \mathbb{N}$ so that $kR_0 \leq r(x) \leq (k+1) R_0$. We have for all $y \in T_j$, $j=1,\cdots,k$,
\begin{displaymath}
G(x,y) \leq c \alpha^{2j} G(x,p)G^\varepsilon(p,y),
\end{displaymath}
where $c,\alpha$ are constants independent of $j$ and $x$.
\end{lemma}

\begin{proof}
By Lemma \ref{spherical standard} there is a positive constant $r_0=r_0(M,R_0)<1$ such that $d(x,T_k)>r_0$. Let $\alpha=\alpha(M,r_0)>1$ be the Harnack constant such that for any $q \in M$ and any positive harmonic function $h$ on $M \setminus \{q\}$,
\begin{displaymath}
h(x) \leq \alpha^{d(x,y)}h(y) \text{ \;\; for all } x,y \in M \setminus B_q(r_0).
\end{displaymath}
It follows that for $y \notin B_x(r_0) \cup B_p(r_0)$, $G(x,y) \leq \alpha^{r(x)} G(p,y)$ and $c_1 G(x,p) \geq  \alpha^{-r(x)}$ with $c_1=c_1(r_0,M)>0$. Thus for $y \in T_k$
\begin{equation}
\label{Harnack exp}
G(x,y) \leq {c_1} \alpha^{2r(x)}G(x,p)G(p,y) \leq {c_1} \alpha^{(2k+2)R_0} G(x,p)G^\varepsilon(p,y).
\end{equation}

Fix $s \in (0,1]$ sufficiently small. Let $\tilde{s}$ and $A>2 R_0$ be constants to be determined later. We will construct a decreasing sequence $C_k=C_k^{(0)} \supset C_k^{(1)} \supset \cdots \supset C_k^{(m)} \supset C_{k-1}$ of cones satisfying

\begin{enumerate}
\item[(\textrm{i})] $d^S(C_k^{(i+1)},M \setminus C_k^{(i)})=s$ on $S_p(kR_0+A-i\tilde{s})$, $i=0,1, \cdots m-1$;
\item[(\textrm{ii})] $0 \leq d^S(C_{k-1},M \setminus C_k^{(m)})<s$ on  $S_p(kR_0+A-m\tilde{s})$;
\item[(\textrm{iii})] $G(x,y) \leq {c_1} \alpha^{2kR_0+2A-i\tilde{s}} G(x,p)G^\varepsilon(p,y)$ for all $y \in C_k^{(i)} \setminus B_p(1)$, $i=0,1, \cdots m$.
\end{enumerate}

Starting from $C_k^{(0)}$, set 
\begin{displaymath}
S_k^{(0)}(r)= C_k^{(0)} \cap S_p(r) \text{ for } r \in (0, \infty)
\end{displaymath}
and 
\begin{displaymath}
S_k^{(1)}(kR_0+A)= \{ q \in S_k^{(0)}(kR_0+A)| d^S(q, M \setminus C_k^{(0)}) > s\}.
\end{displaymath}
Now we can define the cone of $S_k^{(1)}(kR_0+A)$ by
\begin{displaymath}
C_k^{(1)}=\{x \in C_k^{(0)} | \gamma_{px} \cap S_k^{(1)}(kR_0+A) \neq \emptyset\},
\end{displaymath} 
where $\gamma_{px}$ is the geodesic ray starting from $p$ that passes through $x$.

Let $T_k^{(1)}=C_k^{(1)}\setminus B_p(1)$. We claim that there is a constant $\tilde{s}=\tilde{s}(M,s)$ such that
\begin{equation}
\label{T_K^1}
G(x,y) \leq {c_1} \alpha^{2kR_0+2A-\tilde{s}} G(x,p)G^\varepsilon(p,y) \text{ for } y \in T_k^{(1)}.
\end{equation}

In fact, when $y \in T_k^{(1)}\setminus B_p(kR_0+A-1)$, by Lemma \ref{spherical growth} we have $$d^S(y,M \setminus C_k^{(0)})\geq s',$$ where $s'$ is a constant depending only on $s$ and $M$. By Lemma \ref{spherical standard} there exists a constant $s''=s''(s',M)$ such that $d(y,M \setminus C_k^{(0)})>s''$. Now it is implied by Lemma \ref{measure comparison} that there is a constant $\delta=\delta(M,s'',\varepsilon)$ such that $\mu_{B_y(s'')}\leq (1-\delta)\mu^\varepsilon_{B_y(s'')}$. From \eqref{Harnack exp} we have for every $z\in \partial B_y(s'') \subset T_k,$
\begin{displaymath}
G(x,z) \leq {c_1} \alpha^{2kR_0+2A} G(x,p)G^\varepsilon(p,z).
\end{displaymath}
Integrating this inequality with respect to $\mathrm{d} \mu_{B_y(s'')}$  we obtain that
\begin{eqnarray}
\label{Harnack induction}
G(x,y) &=& \int G(x,z) \, \mathrm{d} \mu_{B_y(s'')}(z) \\
 &\leq &  {c_1} \alpha^{2kR_0+2A} \int G(x,p) \nonumber G^\varepsilon(p,z) \mathrm{d} \mu_{B_y(s'')}(z)\\
&\leq &  {c_1} \alpha^{2k R_0+2A} (1-\delta) G(x,p)  \int  G^\varepsilon(p,z) \mathrm{d} \mu^\varepsilon_{B_y(s'')}(z) \nonumber \\
&=& {c_1} \alpha^{2k R_0+2A} (1-\delta) G(x,p) G^\varepsilon(p,y). \nonumber 
\end{eqnarray}

On the other hand, when $y \in T_k^{(1)}\cap B_p(kR_0+A-1)$, by the Harnack inequality we have
$G(x,y) \leq G(x,p) \alpha^{r(y)}$ and $c_1 G(y,p) \geq\alpha^{-r(y)}$. Thus
\begin{equation}
\label{Harnack -1}
G(x,y) \leq {c_1} \alpha^{2r(y)}G(x,p)G(p,y) \leq {c_1} \alpha^{(2kR_0+2A-2)}G(x,p)G^\varepsilon(p,y).
\end{equation}
Combining \eqref{Harnack induction} and \eqref{Harnack -1}, the inequality \eqref{T_K^1} is proved for $\tilde{s}=\min\{-\dfrac{\log(1-\delta)}{\log\alpha},1\}$.

Now assume that we have constructed inductively $C_k^{(0)},\cdots,C_k^{(i)}$, set
\begin{displaymath}
S_k^{(i)}(r)= C_k^{(i)} \cap S_p(r) \text{ for } r \in (0, \infty)
\end{displaymath}
and 
\begin{displaymath}
S_k^{(i+1)}(kR_0+A-i\tilde{s})= \{ q \in S_k^{(i)}(kR_0+A-i\tilde{s})| d^S(q, M \setminus C_k^{(i)}) > s\}.
\end{displaymath}
Define $C_k^{i+1}$ to be the cone of $S_k^{(i+1)}(kR_0+A-i\tilde{s})$, i.e.,
\begin{displaymath}
C_k^{(i+1)}=\{x \in C_k^{(i)} | \gamma_{px} \cap S_k^{(i+1)}(kR_0+A-i\tilde{s}) \neq \emptyset\}.
\end{displaymath} 
it is easy to see the sequence $C_k^{(i)}$ satisfies condition $(\textrm{i})$,

Applying the argument in the proof of \eqref{T_K^1} to the cone $C_k^{(i+1)}$ we obtain inductively
\begin{equation}
\label{T_k^i}
G(x,y) \leq {c_1} \alpha^{2kR_0+2A-(i+1)\tilde{s}} G(x,p)G^\varepsilon(p,y) 
\end{equation}
for all $y \in C_k^{(i+1)} \setminus B_p(1)= T_k^{(i+1)}.$

We can continue this procedure provided $d^S(C_{k-1}, M \setminus C_k^{(i)}) \geq s$ on $B_p(kR_0+A-i\tilde{s})$. Eventually we get $m \in \mathbb{N}$ so that 
\begin{equation}
\label{spherical k}
0 \leq d^S(C_{k-1},  M \setminus C_k^{(m)}) < s \text{ on } B_p(kR_0+A-m\tilde{s}).
\end{equation}

If $A-m\tilde{s}<0$, then from \eqref{T_k^i} we obtain that
\begin{eqnarray*}
G(x,y) &\leq& {c_1} \alpha^{2kR_0+2A-m\tilde{s}} G(x,p)G^\varepsilon(y,p) \\
\nonumber &\leq& {c_1} \alpha^{(2k-2)R_0+2A} G(x,p)G^\varepsilon(y,p)
\end{eqnarray*}
for all $y \in C_k^{(m)} \setminus B_p(1) \subset T_{k-1}$.

If $A-m\tilde{s}>0$, using Lemma \ref{sector} we can estimate the spherical distance on $S_p(kR_0+A-m\tilde{s})$ as follows,
\begin{eqnarray}
\label{spherical sum}
d^S(C_{k-1}, M \setminus C_k^{(m)}) & \geq & d^S( C_{k-1}, M \setminus C_k) - d^S(C_k^{(m)},M \setminus C_k) \\
\nonumber & \geq & d^S( C_{k-1}, M \setminus C_k) - \Sigma_{i=1}^m d^S(C_k^{(i)},M \setminus C_k^{(i-1)})\\
\nonumber &\geq & \eta (A-m\tilde{s})-ms.
\end{eqnarray}
Combining \eqref{spherical k} and \eqref{spherical sum} we can choose $A$ large enough so that $m \geq 2 R_0/\tilde{s}$, it then follows from \eqref{T_k^i} that
\begin{equation*}
G(x,y) \leq {c_1} \alpha^{2kR_0+2A-m\tilde{s}} G(x,p)G^\varepsilon(p,y) \leq {c_1} \alpha^{(2k-2)R_0+2A} G(x,p)G^\varepsilon(p,y) 
\end{equation*}
for all $y \in C_k^{(m)} \setminus B_p(1) \subset T_{k-1}$. By repeating the argument above inductively for $j=k-1,k-2,\cdots,1$, the Lemma is proved for $c=c_1 \alpha^{2A}$.

\end{proof}

By Theorem \ref{density}, there exists an axis $\gamma$ of a hyperbolic axial isometry $\phi \in D$ such that $\gamma(-\infty),\gamma(\infty) \in T_p(\upsilon,\theta,1)$. We can choose $N\in \mathbb{N}$ large enough so that $\phi^N(M \setminus T_p(\upsilon,\theta,1)) \subset T_p(\upsilon,\theta,1)$. Set
$p_1=\phi^N(p)$, $\upsilon_1=\phi^N(\upsilon)$.
Since Green's function is invariant under isometries, we have
\begin{displaymath}
G(y,z) \leq c G(y,p_1)G^\varepsilon(p_1,z)
\end{displaymath}
for all $y \in T_{p_1}(-\upsilon_1,\pi/2,1)$ and $z \in T_{p_1}(\upsilon_1,\theta,1)$, where the constant $c$ is  as in Proposition \ref{proposition}.

We now proceed to establish the boundary Harnack inequality.

\begin{theorem}
\label{Harnack final}
For all $x\in T_p(-\upsilon,\pi/2,1)$ and all $y \in T_{p_1}(-\upsilon_1,\pi/2,1)$,
\begin{displaymath}
G(x,y) \leq cG(x,p)G(p,y),
\end{displaymath}
where $c$ is a constant independent of $x,y$.
\end{theorem}

\begin{proof}
The proof is essentially the same as that of Theorem 1 in \cite{MR890161}. For the sake of completeness we sketch the proof.

Let $\mu_x$ be the harmonic measure supported on $\partial T_p(v,\theta,1)$. From the choice of $\phi^N$ we have $\partial T_p(v,\theta,1) \subset \phi^N(T_p(v,\theta,1))=T_{p_1}(v_1,\theta,1)$. By Proposition \ref{proposition} we have
\begin{eqnarray}
\label{harmonic integral}
G(x,y) &=& \int_{\partial T_p(v,\theta,1) } G(y,z) \mathrm{d} \mu_x(z)\\
&\leq& c G(y,p_1) \int_{\partial T_p(v,\theta,1)} G^\varepsilon(p_1,z)\mathrm{d} \mu_x(z) \nonumber
\end{eqnarray}

Since $G^\varepsilon(p_1,\cdot)$ is a $\Delta$-potential, there is a positive measure $\lambda$ on $M$ such that $G^\varepsilon(p_1,z)=G(\lambda)(z)$. By the Fubini theorem \eqref{harmonic integral} becomes
\begin{equation}
\label{Harnack integral}
G(x,y) \leq cG(y,p_1) \int_M G(\mu_x) \mathrm{d} \lambda.
\end{equation}

It follows from the definition of reduction and the maximum principle that
\begin{displaymath}
G(\mu_x)(z) \leq c G(x,p)G^\varepsilon(p,z) \text{ on } M,
\end{displaymath}
Together with \eqref{Harnack integral} we obtain
\begin{equation*}
G(x,y) \leq c^2 G(y,p_1)G(x,p) \int_M G^\varepsilon(p,z)\mathrm{d} \lambda.
\end{equation*}

By the Harnack inequality, $G(y,p_1) \leq c'G(p,y)$ for a constant $c'=c'(r(p_1),M)$. It remains to prove that $\int_M G^\varepsilon(p,z)\mathrm{d} \lambda$ is bounded. In fact, from \eqref{Green comparison} we have
\begin{equation*}
\lambda=-\Delta G^\varepsilon_{p_1}=\delta_{p_1}+\varepsilon G^\varepsilon_{p_1},
\end{equation*}
and it follows that
\begin{eqnarray*}
\int_M G^\varepsilon(p,z)\mathrm{d} \lambda(z) &=& G^\varepsilon(p,p_1)+\varepsilon G^\varepsilon(G^\varepsilon_{p_1})(p)
\end{eqnarray*}
is bounded by a constant depending only on $\varepsilon,r(p_1)$ and $M$. This completes the proof.
\end{proof}

\section{Proof of Theorem \ref{main}}

Let $g_t$ be the geodesic flow on the unit tangent bundle $SM$ and $\pi:SM \to M$ be the projection map.

To prove Theorem \ref{main}, we use the following reformulation of Theorem \ref{Harnack final}: there is an open set $V \subset SM$, such that for all $\omega \in V$ and $x \in T(-\omega,\theta,1), y \in T(\omega,\theta,1)$ we have
\begin{equation}
\label{reformulation}
G(x,y) \leq c G(x, \pi(\omega))G(\pi(\omega),y),
\end{equation}
where $\theta,c$ are independent of the choices of $x,y$ and $\omega$.

In fact, let $\omega_0 \in S_p$ be a unit tangent vector at $p$ such that $\gamma_{\omega_0}(-\infty) \in T_p(-\upsilon,\pi/2,1)$ and $\gamma_{\omega_0}(\infty) \in T_{p_1}(-\upsilon_1,\pi/2,1)$. Then we can choose $V$ to be a small neighborhood of $\omega_0$ in $SM$ and \eqref{reformulation} follows.

In addition, up to an open subset we may assume that for all $\omega \in V$,
\begin{equation}
\label{equi-hyperbolic}
C(g_T(\omega),\theta+\theta_0) \subset C(\omega,\theta),
\end{equation}
where $T,\theta_0$ are independent of $\omega$.

For $i \in \mathbb{N}$, set
\begin{equation*}
V_i=\{w \in S_p M |\; \exists t \geq i \text{ such that } g_t(w) \in DV\}.
\end{equation*}
Each $V_i$ is an open dense set in $S_p M$ and therefore $\bigcap_{i=0}^\infty V_i$ is a generic set. Let $V_i(\infty)=\{\gamma_\omega(\infty)|\omega \in V_i\}$. We prove that $R=\bigcap_{i=0}^\infty V_i(\infty)$ is the required generic set in Theorem \ref{main}.

For any $\omega \in \bigcap_{i=0}^\infty V_i$, we can find an infinite sequence of real numbers $t_0 \leq t_1 \leq \cdots$ so that $t_i \geq i$ and $g_{t_i}(\omega) \in DV$. Let $x_i=\pi(g_{t_i}(\omega))$, then $x_i \to \gamma_\omega(\infty)$.  From \eqref{reformulation}  we have
\begin{equation}
\label{Harnack T_i}
G(x,y) \leq c G(x,x_i) G(x_i,y)
\end{equation}
for all $x \in T_{x_i}(-g_{t_i}(\omega),\theta,1), y \in T_{x_i}(g_{t_i}(\omega),\theta,1)$.

Recall that every fundamental sequence $Y=\{y_k\}$ in $M$ corresponds to  a limiting harmonic function
\begin{displaymath}
h_Y(x)=\lim_{k \to \infty} \dfrac{G(x,y_k)}{G(p,y_k)} \text{ with } h_Y(p)=1.
\end{displaymath}

By the Arzela-Ascoli Theorem, up to a subsequence $h_{x_k}$ converge to a nonnegative harmonic function $h_\omega$. We want to prove that all the limiting functions of fundamental sequences converging to $\gamma_\omega(\infty)$ are identical to $h_\omega$. To this end, let $Y=\{y_k\}$ be an arbitrary fundamental sequence in $M$ converging to $\gamma_\omega(\infty)$. Given $i \in \mathbb{N}$, it follows from \eqref{Harnack T_i} that for $k$ sufficiently large,
\begin{equation}
\label{double 1}
c^{-1}G(x,x_i) \leq \dfrac{G(x,y_k)}{G(x_i,y_k)} \leq c G(x,x_i)
\end{equation}
on $T_{x_i}(-g_{t_i}(\omega),\theta,1)$. In particular, for $x=p$,
\begin{equation}
\label{double 2}
c^{-1}G(p,x_i) \leq \dfrac{G(p,y_k)}{G(x_i,y_k)} \leq c G(p,x_i).
\end{equation}

Combining \eqref{double 1} and \eqref{double 2} we obtain
\begin{equation*}
c^{-2}\dfrac{G(x,x_i)}{G(p,x_i)} \leq \dfrac{G(x,y_k)}{G(p,y_k)} \leq c^2 \dfrac{G(x,x_i)}{G(p,x_i)}
\end{equation*}
on $T_{x_i}(-g_{t_i}(\omega),\theta,1)$. Taking the limit as $k \to \infty$ we obtain
\begin{equation}
\label{c^-2}
c^{-2}\dfrac{G(x,x_i)}{G(p,x_i)} \leq h_Y(x) \leq c^2 \dfrac{G(x,x_i)}{G(p,x_i)}
\end{equation}
for all $x \in T_{x_i}(-g_{t_i}(\omega),\theta,1)$. Observe that for $h_\omega$ we have the same estimate, hence
\begin{equation}
\label{equi}
c^{-4} \leq \dfrac{h_Y(x)}{h_\omega(x)} \leq c^4 \text{ on } M=\bigcup_{i=0}^\infty T_{x_i}(-g_{t_i}(\omega),\theta,1) .
\end{equation}

Now consider the reduction $u_{k}$ of $h_Y$ on $B_{x_k}(1)$, i.e.,
\begin{equation*}
u_k=\inf \{s:M \to \mathbb{R}~| ~s>0 \text{ and superharmonic on } M; s \geq h_Y \text{ on } B_{x_k}(1) \}.
\end{equation*} 
 By \eqref{c^-2} and the Harnack inequality we have for $z \in \partial B_{x_k}(1)$, 
\begin{equation*}
u_k(z) \geq h_Y(z) \geq c_1 \dfrac{1}{G(p,x_k)},
\end{equation*}
where $c_1>0$ depends only on $M$ and $c$. On the other hand,
\begin{equation*}
h_{x_k}(z)=\dfrac{G(z,x_k)}{G(p,x_k)} \leq c_2 \dfrac{1}{G(p,x_k)}
\end{equation*}
on $\partial B_{x_k}(1)$, where $c_2$ depends only on $M$. Since on $M \setminus B_{x_k}(1)$, $u_k$ is superharmonic and $h_{x_k}$ is harmonic, by the maximum principle we obtain
\begin{equation*}
h_Y(x) \geq u_k(x) \geq c_3 h_{x_k}(x)
\end{equation*}
on $M \setminus B_{x_k}(1)$, where $c_3$ depends only on $M, c$. Taking $k \to \infty$ yields
\begin{equation*}
h_Y(x) \geq \lim_{k\to \infty} \inf u_k(x) \geq c_3 h_{\omega}(x) \text{ on } M,
\end{equation*}
together with \eqref{equi} we have
\begin{equation*}
\lim_{k\to \infty} \inf u_k(x) \geq c_4 h_Y
\end{equation*}
on $M$ with $c_4=c_3 c^{-4}>0$.

Let
\begin{equation*}
\alpha_0=\sup \{ \alpha | h_\omega -\alpha h_Y \geq 0 \text{ on } M\},\;\; \beta_0 = \inf \{ \beta| h_\omega - \beta h_Y \leq 0 \text{ on } M \}.
\end{equation*}
We have $0 \leq \alpha_0 \leq \dfrac{h_\omega}{h_Y} \leq \beta_0$ on $M$.

To prove the theorem, it is sufficient to show that $\alpha_0=\beta_0$. If $\alpha_0 < \beta_0$, by replacing $h_\omega$ by a linear combination of $h_\omega$ and $h_Y$, we may assume that $\alpha_0=0$ and $\beta_0=1$.

For $z \in \partial B_{x_k}(1)$, by the Harnack inequality we have
\begin{equation*}
h_\omega(z) \geq c_5 h_\omega(x_k), \;\; u_k(z) \leq c_6 h_Y(x_k),
\end{equation*}
where $c_5, c_6$ depend only on $M$. It then follows from the maximum principle that
\begin{equation*}
h_\omega \geq c_7 u_k \dfrac{h_\omega(x_k)}{h_Y(x_k)}
\end{equation*}
on $M \setminus B_{x_k}(1)$, where $c_7=c_5/c_6$. Taking the limit as $k \to \infty$ we have
\begin{eqnarray*}
h_\omega &\geq& c_7   \lim_{k\to \infty} \inf u_k \lim_{k\to \infty} \sup \dfrac{h_\omega(x_k)}{h_Y(x_k)} \\
&\geq& c_4 c_7 h_Y \lim_{k\to \infty} \sup \dfrac{h_\omega(x_k)}{h_Y(x_k)}
\end{eqnarray*}
on $M$. Since $\alpha_0=\sup \{ \alpha | h_\omega -\alpha h_Y \geq 0 \text{ on } M\}=0$, we must have $\lim_{k\to \infty} \sup \dfrac{h_\omega(x_k)}{h_Y(x_k)}=0$.

On the other hand, if we repeat the argument above for the positive function $g=h_Y-h_\omega$, we have
\begin{equation*}
\sup \{ \alpha | g -\alpha h_Y \geq 0 \text{ on } M\}=0; \;\; \inf \{ \beta| g - \beta h_Y \leq 0 \text{ on } M \},
\end{equation*}
and
\begin{eqnarray*}
g &\geq& c_4 c_7 h_Y \lim_{k\to \infty} \sup \dfrac{g(x_k)}{h_Y(x_k)}  \\
&=& c_4 c_7 h_Y \lim_{k\to \infty} \sup \dfrac{h_Y(x_k)-h_\omega(x_k)}{h_Y(x_k)}  \\
&=& c_4 c_7 h_Y \lim_{k\to \infty} \sup (1-\dfrac{h_\omega(x_k)}{h_Y(x_k)}),
\end{eqnarray*}
which implies $\lim_{k\to \infty} \sup (1-\dfrac{h_\omega(x_k)}{h_Y(x_k)})=0$, or equivalently, $\lim_{k\to \infty} \inf \dfrac{h_\omega(x_k)}{h_Y(x_k)}=1$. Contradiction!

Therefore $h_Y$ must be a multiple of $h_\omega$. Since $h_Y(p)=h_\omega(p)=1$, we have $h_Y=h_\omega$ on $M$. Hence there is a unique limiting function $h_\xi$ for all fundamental sequences converging to $\xi \in R$. It follows from \eqref{equi-hyperbolic} that $h$ vanishes on $M(\infty) \setminus \{\xi\}$. This completes the proof of Theorem \ref{main}.

\bibliographystyle{amsalpha}
\bibliography{rank-1} 

\Addresses

\end{document}